\documentclass[a4paper,10pt]{amsart}
\usepackage[utf8x]{inputenc}
\usepackage{amsmath, amssymb, amsthm, amsfonts}
\usepackage{fullpage}
\usepackage{times}
\usepackage{times,mathptmx}
\usepackage{tikz}
\usepackage{mathtools}
\usepackage{fontenc}
\usepackage{enumerate}
\usepackage{yfonts}

\title{On Additive Representation Functions}
\author{R. Balasubramanian and Sumit Giri}
\date{}
\newtheorem{theoremm}{Theorem}
\newtheorem{lemma}{Lemma}
\newtheorem{theorem}{Theorem}

\newtheorem{rem}{Remark}

\newtheorem{cor}{Corollary}

\newcommand{\N}{\mathbb{N}}

\newcommand{\R}{\mathbb{R}}

\newcommand{\A}{\mathcal{A}}

\newcommand{\s}{\displaystyle \sum}
\begin{document}

\begin{abstract}
Let $\A=\{a_1<a_2<a_3 \cdots <a_n< \cdots\}$ be an infinite sequence of non-negetive integers and let $R_2(n)=|\{(i,j):\ \ a_i+a_j=n;\ \ a_i,a_j\in \A;\ \ i\le j\}|$. We define $S_k=\s_{l=1}^k(R_2(2l)-R_2(2l+1))$. We prove that if $L^{\infty}$ norm of $S_k^+(=\max\{S_k,0\})$ is small, then $L^1$ norm of $\frac{S_k^+}{k}$ is large.
\end{abstract}
\maketitle
\section{\textbf{Introduction}}
Let $\A=\{a_1,a_2,\cdots\}(0\le a_1< a_2< \cdots)$ be an infinite sequence of non-negative integers. For $n\in \N_0$, Define
\begin{align}
R_1(n)=&R_1(\A,n)=\s_{a_i+a_j=n}1\\
 R_2(n)=&R_2(\A,n)=\s_{\underset{i\le j}{a_i+a_j=n}}1.\label{eq:2}
 \end{align}

Now, it is easy to check that if $\A$ is a full set or a complement of a finite set inside the set of natural numbers, then $R_1$ and $R_2$ are monotonically increasing. Here we are interested in inverse problems, i.e., how the monotonicity of the representation functions affects the cardinality of the set $\A$.\par
The question of characterisation of the set $\A$, under the condition that either $R_1(n)$ or $R_2(n)$ is monotonic, was raised by Erd\H{o}s, S\'{a}rk\"{o}zy and S\'{o}s \cite{8}. Also see \cite{9} and \cite{2}.

In \cite{8}, Erdős, Sárközy and Sós proved that $R_1(n)$ can be monotonically increasing from a certain point, only in a trivial way. See \cite{8} and \cite{2} for the following theorem. 
\begin{theoremm}\label{Th_A}
  If $R_1(n+1)\ge R_1(n)$ for all large $n$, then $\N\setminus \A$ is a finite set. 
\end{theoremm}

While analogous conclusion is not true in case of $R_2$.

 If, we define \begin{align}\A(N)&=|\A \cap [1,N] |,\label{eq:4}\end{align} then the first author \cite{2} proved the following theorem:
\begin{theoremm}\label{Th_B}
If $R_2(n+1)\ge R_2(n)$ for all large $n$, then $\A(N)=N+O(\log N)$.
\end{theoremm}
In other words, If $R_2(n)$ is monotonic, then the complement set of $\A$ is almost of order $O(\log N)$.\par\medskip
In the first part of this paper we shall focus on the function $R_2$ and quantities related to monotonicity of it. Also in \emph{Section 6}, we shall make a remark concerning a question raised by S\'{a}rk\"{o}zy \cite{10}, related to monotonicity of $R_1$.\\

In \cite{9} Erd\H{o}s, S\'{a}rk\"{o}zy and S\'{o}s  proved
\begin{theoremm}\label{Th_D}
 If \begin{align}\underset{n\to +\infty}{\lim}\frac{n-\A(n)}{\log n}&=+\infty, \label{eq:5}\end{align} then we have, \begin{align} \underset{N\to +\infty}{\limsup}\s_{k=1}^{N}(R_2(2k)-R_2(2k+1))&=+\infty.\label{eq:6}\end{align}
\end{theoremm}
The assumption (\ref{eq:5}) in the above theorem can not be relaxed. In fact Erd\H{o}s, S\'{a}rk\"{o}zy and S\'{o}s \cite{9} constructed a sequence $\hat{\A}$ where $(n-\hat{\A}(n))>c\log n$ (for large $n$ and fixed constatnt $c$) and $$\underset{N\to +\infty}{\limsup}\s_{k=1}^N(R_2(2k)-R_2(2k+1))<+\infty.$$\\
In \cite{4}, Tang and Chen gave a quantative version of Theorem \ref{Th_D}. Before we state their theorem, let us define a few notations.
\begin{align}
 S_n&=\s_{k\le n}(R_2(2k)-R_2(2k+1)),\nonumber\\
m(N)&=N(\log N+\log \log N). \nonumber
\end{align}
Also $L^{\infty}$ norm of $S_n$, denoted by $T(N)$, is defined as follow: \begin{align}T(N)&=\max_{n\le m(N)}S_n=\max_{n\le m(N)}\s_{k\le n}(R_2(2k)-R_2(2k+1)).\label{eqn:9}\end{align}
In \cite{4} the authors proved that, when the ratio $\frac{T(N)}{\A(N)}$ is bounded above by a small enough fixed constatnt, then $T(N)$ and $\frac{N-A(N)}{\log N}$ satisfies a simple inequality. More precisely,  
\begin{theoremm}\label{Th_E}
If $T(N)$ be defined as in (\ref{eqn:9}) and 
 \begin{align} T(N)&<\frac{1}{36}\A(N),\label{eqn:10} \end{align}
for all large enough $N$, then there exists a $C>0$, depending only on $\A$, such that
 \begin{align}T(N)&>\frac{1}{80e}\frac{N-\A(N)}{\log N}-\frac{11}{4}-\frac{C}{8}\label{eqn:11}
 \end{align}
for all large enough $N$. 
\end{theoremm}
It is easy to see that, under the condition (\ref{eqn:10}), Theorem \ref{Th_E} implies Theorem \ref{Th_D}.\\
Now, set \begin{align}
         S^+_n&=\max\{S_n,0\},\nonumber
\shortintertext{and} T^+(N)&=\underset{n\le m(N)}{\max}\{S^+_n\}\nonumber.
        \end{align}
\textbf{Note: }$T(N)$ and $T^+(N)$ are same unless all the elements of the set $\{S_n\,:\, n\le m(N)\}$ are negative.\\

In this paper, we again assume that $\frac{T(N)}{\A(N)}$ is bounded above and prove an improved version of (\ref{eqn:11}) where we replace $L^{\infty}$ norm of $S(n)$ by $L^1$ norm of $\frac{S^+(n)}{n}$. More precisely, we prove the following theorem: 
\begin{theorem}\label{Th_1}
 Let $\A$ be an infinite sequence of positive integers and there exists $N_0$ such that $T(N)<\frac{1}{36}\A(N)$ for all $N\ge N_0$. Then there exists a constatnt $c_1>0$, depending on $\A$, such that 
\begin{align}
 \s_{n=1}^{m(N)}\frac{S^+_n}{n}&>\frac{1}{10e}(N-\A(N))-\frac{1}{4}\log N-c_1. \label{eq_3}
\end{align}
for all large enough $N$.
\end{theorem}
 
\begin{cor}\label{cor_1}
 If (\ref{eqn:10}) in \emph{Theorem \ref{Th_E}} holds, then for any $\epsilon>0,$ \begin{align}T^+(N)&>\frac{1}{10e+\epsilon}\frac{N-\A(N)}{\log N}-\frac{1}{4}.\label{eqn: 15}\end{align} for any large enough $N$.
\end{cor}
So, if at least one of $S(n)$ is non-negetive, then $T^+(N)$ indeed equals $T(N)$. In that case, \emph{Corollary \ref{cor_1}} gives \emph{Theorem \ref{Th_E}} with a better constant. \emph{Corollary \ref{cor_1}} also implies the following:

\begin{cor}\label{cor_3} 
If $\underset{N\to +\infty}{\limsup}\, 
\frac{N-\A(N)}{\log N}=+\infty$, then  \begin{align}\underset{n\to +\infty}{\limsup}\,\{ S_n\}&=+\infty.\label{eqn:17}\end{align} 
\end{cor}

\section{\emph{Generating Functions:}}
It is more natural to consider the problem in terms of generating function.\par

Set  $$f(z)=\s_{a\in \A}z^a,\quad  |z|<1.$$ 
Then, $$f(z)^2=\s_{n=1}^{+\infty}R(n)z^n.$$
For any positive real number $Y$, define  \begin{align}\psi(Y)=f(e^{-\frac{1}{Y}})&=\s_{a\in \A}e^{-\frac{a}{Y}},\label{eqn:19}\end{align}
and \begin{align}g(Y)&=1+4(1-e^{-\frac{2}{Y}})\s_{k=1}^{+\infty} S_ke^{-\frac{2k}{Y}}.\label{eqn:20}\end{align}

\begin{theorem}\label{Th_2}
Let $g(Y)$ and $\psi(Y)$ be defined as above. Also assume \begin{align}g(Y)\le \min \{\psi\left(\frac{Y}{2}\right),\frac{1}{9}Y \}\label{eqn:13}\end{align} for all sufficiently large positive real number $Y$. Then $$\psi(Y)\ge Y\exp\left(-\left(\frac{2.3}{2Y}\left(\log_2Y+\frac{16}{Y}\sum_{k=1}^{\infty}S_k^+\frac{e^{-\frac{2k}{Y}}}{1-e^{-\frac{2k}{Y}}}\right)\right)-\frac{c}{Y}\right)$$ 
for some positive constatnt $c$ depending only on first few elements of $\A$.
\end{theorem}
In \emph{Section 3}, we will give a proof of \emph{Theorem \ref{Th_2}}. In \emph{Section 4}, we will show how \emph{Theorem \ref{Th_1}} follows from \emph{Theorem \ref{Th_2}}.
\section{\emph{Notations and prelimanary lemmas:}}

Consider a function $h:\R \mapsto [0, +\infty)$. For any real number $Y$ and integer $\alpha \ge 0$ define $H(Y;\alpha)$ by recurrence, as follows:
\begin{align} H(Y;0)&=0\nonumber\\
 H(Y;\alpha)&=\frac{h(Y)}{2}+\frac{h(\frac{Y}{2})}{4}+\frac{h(\frac{Y}{4})}{8}+\cdots +\frac{h(\frac{Y}{2^{\alpha-1}})}{2^\alpha}\nonumber\\
 &=\sum_{j=0}^{\alpha -1}\frac{1}{2^{j+1}}h\left(\frac{Y}{2^j}\right)\quad\text{for integer }\alpha\ge 1.\label{eq:13}\end{align}
 Also \begin{align}H(Y)=\sum_{j=0}^{\infty}\frac{1}{2^{j+1}}h\left(\frac{Y}{2^j}\right).\label{eqn:14}\end{align}
\begin{lemma}\label{lm_2}
 If $h(Y)$ and $H(Y;\alpha)$ is defined as above and \begin{align}(\psi(Y))^2&\ge 2Y\exp(-h(Y))\psi(\frac{Y}{2})\label{eqn:24}\end{align} for all real number $Y\ge \tilde{N}_0$, then for every integer $\alpha\ge 0$, \begin{align}\psi(Y)&\ge Y\exp(-H(Y;\alpha))\left (\frac{\psi \left (\frac{Y}{2^{\alpha}}\right)2^{\alpha}}{Y}\right)^{\frac{1}{2^\alpha}}\label{eqn:25}\end{align} for any real number $Y\ge 2^\alpha\tilde{N}_0.$
\end{lemma}
\begin{proof}
 For $\alpha=0$, both sides are equal.\\
 For the general case, suppose it is true for $\alpha = \alpha_0$. Then   \begin{align*}
                                                                          (\psi(Y))^2&\ge 2Y\exp(-h(Y))\psi\left(\frac{Y}{2}\right)\\
                                                                          &=Y^2\exp\left(-h(Y)-H\left(\frac{Y}{2};\alpha_0\right)\right)\left(\frac{\psi\left(\frac{Y}{2^{\alpha_0+1}}\right)2^{\alpha_0+1}}{Y}\right)^{\frac{1}{2^{\alpha_0}}}\quad \text{for $Y\ge 2N_1$ }\\
                                                                          &=\left(Y\exp(-H(Y,\alpha +1))\left(\frac{\psi\left(\frac{Y}{2^{\alpha_0+1}}\right)2^{\alpha_0+1}}{Y}\right)^{\frac{1}{2^{\alpha_0+1}}}\right)^2
                                                                         \end{align*}
and hence the result.
\end{proof}
\begin{lemma}\label{lm_3}
 There exist a $c>0$ such that, if $Y$ is larghe enough, then we have $$\left(\frac{\psi\left(\frac{Y}{2^{\alpha}}\right)2^{\alpha}}{Y}\right)^{\frac{1}{2^\alpha}}\ge \exp(-\frac{c}{Y})$$ for some $\alpha \le \log_2Y$.
\end{lemma}
\begin{proof}
 Now fix an interval $[a,2a]$ so that $\psi(a)\ge 1$.\\
Then choose $\alpha$ suitably so that $\frac{Y}{2^\alpha}\in [a,2a]$. In that case, we have \begin{align*}\left(\frac{\psi(\frac{Y}{2^\alpha})2^{\alpha}}{Y}\right)^{\frac{Y}{2^\alpha}}\ge \left( \frac{1}{2a}\right)^{2a}.\end{align*}
This proves the lemma.
\end{proof}

\begin{lemma}\label{lm_1}
 Let $0<x< 1$ be a real number. Then $\s_{n=0}^{+\infty}2^nx^{2^n}\le \frac{2x}{(1-x)}$.
\end{lemma}
\begin{proof}
 Note that $$2^nx^{2^n}\le 2\s_{2^{n-1}<j\le 2^n}x^j.$$
 Summing over $n=1$ to $+\infty$, $$\s_{n=1}^{+\infty}2^nx^{2^n}\le 2\s_{j=2}^{+\infty}x^j=\frac{2x^2}{1-x}.$$
 Adding $x$ (corresponding to $n=0$) on both sides, we get the result.
\end{proof}
\begin{lemma}\label{lm_4}
 In the notaion of \emph{Lemma \ref{lm_2}}, let $h(Y)=d\frac{g(Y)}{Y}$ for some fixed positive constant $d$, to be choosen later. Then $$H(Y,\alpha)\le \frac{d}{2Y}\left(\alpha+\frac{8}{Y}\sum_{k=1}^{\infty}S_k^+\frac{e^{-\frac{2^{k+1}}{Y}}}{1-e^{-\frac{2^{k+1}}{Y}}}\right)$$
\end{lemma}
\begin{proof}
 set $x=e^{-\frac{2}{Y}}$.
 \begin{align*}
g(Y)&=1+4(1-e^{-\frac{2}{Y}})\sum_{k=1}^{\infty}S_k^+x^k.\\
&\le 1+\frac{8}{Y}\sum_{k=1}^{\infty}S_k^+x^k.
 \end{align*}
Then 
\begin{align*}
 H(Y;\alpha)&=\sum_{j=0}^{\alpha-1}\frac{1}{2^{j+1}}h\left(\frac{Y}{2^j}\right)\\
 &\le\sum_{j=0}^{\alpha-1}\frac{d}{2Y}\left(1+\frac{8}{Y}\sum_{k=1}^{\infty}S_k^+2^jx^{2^jk}\right)\\
 &\le \frac{d}{2Y}\left(\alpha+\frac{8}{Y}\sum_{k=1}^{\infty}S_k^+\sum_{j=0}^{\infty}2^jx^{2^jk}\right)\\
 &=\frac{d}{2Y}\left(\alpha+\frac{8}{Y}\sum_{k=1}^{\infty}S_k^+\frac{2x^k}{1-x^k}\right).
\end{align*}
\end{proof}

\section{\textbf{Proof of Theorem \ref{Th_2}:}}

It is easy to verify the following inequality by comparing the coefficients of $z^n$ from both sides. 
\begin{align}f(z^2)&=\frac{1-z}{2z}(f(z))^2+2\s_{k=1}^{+\infty} (R_2(2k)-R_2(2k+1))z^{2k}-\frac{(1+z)}{2z}f(-z)^2.\label{eqn:28}\end{align}
                                                                                                                                                       
If $z>0$, this gives, 
\begin{align}f(z^2)&\le \frac{1-z}{2z}f(z)^2+2\s_{k=1}^{+\infty} (R_2(2k)-R_2(2k+1))z^{2k}.\label{eqn:29}\end{align}
Now, considering the right hand side of the summation, we get 
\begin{align*}
 \s_{k=1}^{+\infty}(R_2(2k)-R_2(2k+1))z^{2k}&=\s_{k=1}^{+\infty}(S_k-S_{k-1})z^{2k}\\
 &=\s_{k=1}^{+\infty}S_k(z^{2k}-z^{2k+2})-S_0z^2\\
 &\le (1-z^2)\s_{k=1}^{+\infty}S_kz^{2k}.
\end{align*}
Thus, from (\ref{eqn:29}) we get 
\begin{align}f(z^2)&\le \frac{1-z}{2z}f(z)^2+2(1-z^2)\s_{k=1}^{+\infty}S_kz^{2k}.\nonumber\end{align}
Now putting $z=e^{\frac{-1}{Y}}$, we get 
\begin{align}\psi\left(\frac{Y}{2}\right)&\le \frac{1}{2}\left(\frac{1}{Y}+\frac{1}{Y^2}\right)(\psi(Y))^2+2(1-e^{-\frac{2}{Y}})\s_{k=1}^{+\infty}S_k e^{-\frac{2k}{Y}}.\nonumber\end{align}
Since $\psi(Y)\le Y$, this gives 
\begin{align} 
  2Y\psi\left(\frac{Y}{2}\right)   & \le (\psi(Y))^2+Yg(Y).\nonumber
  \end{align}
Thus, \begin{align}
            (\psi(Y))^2& \ge 2Y \psi\left(\frac{Y}{2}\right)-Yg(Y).\label{eqn:33}
           \end{align}
\begin{lemma}\label{lm_6}
 If $g(Y)\le \psi\left(\frac{Y}{2}\right)$, then for all large enough real number $Y$ $$\psi(Y)\ge 0.49Y. $$
\end{lemma}
\begin{proof}
 
 Since $g(Y)\le \psi\left(\frac{Y}{2}\right)$, using (\ref{eqn:33}) we get $$(\psi(Y))^2\ge Y\psi\left(\frac{Y}{2}\right).$$
 Then (\ref{eqn:24}) in \emph{Lemma \ref{lm_2}} is holds with $h(Y)=\log 2$.\\
 In that case $$H(Y)=\sum_{0\le j\le \infty}\frac{1}{2^{j+1}}h\left(\frac{Y}{2^j}\right)= \log 2.$$\\
 
This gives, by \emph{Lemma \ref{lm_2}} and \emph{Lemma \ref{lm_3}}, $$\psi(Y)\ge (0.49)Y$$ if $Y$ is large enough.
 
 \end{proof}
Thus, combining (\ref{eqn:33}) and \emph{Lemma \ref{lm_6}} we get 
\begin{align}\psi(Y)^2&\ge 2Y\psi\left(\frac{Y}{2}\right)\left(1-\frac{g(Y)}{0.49 Y}\right)\label{eqn:39}\end{align}
 for sufficiently large $Y$.\\
Since $\frac{g(Y)}{Y}<\frac{1}{9}$, equation (\ref{eqn:24}) in \emph{Lemma \ref{lm_2}} is satisfied with $h(Y)=2.3\frac{g(Y)}{Y}$. \\
Hence \emph{Lemma \ref{lm_4}} and \emph{Lemma \ref{lm_2}} together give the following:

\begin{align}
 \psi(Y)&\ge Y\exp\left(-\frac{2.3}{2Y}\left(\alpha+\frac{16}{Y}\sum_{k=1}^{\infty}S_k^+\left(\frac{e^{-\frac{2k}{N}}}{1-e^{-\frac{2k}{N}}}\right)\right)\right)\left(\frac{\psi\left(\frac{Y}{2^\alpha}\right)2^{\alpha}}{Y}\right)^{\frac{1}{2^\alpha}}.\label{eqn:50}
\end{align}

Hence \emph{Theorem \ref{Th_2}} follows from (\ref{eqn:50}) and \emph{Lemma \ref{lm_3}}.

\section{Proof of Theorem \ref{Th_1}}
\begin{lemma}\label{lm_5}
 Let $g(N)$ and $T(N)$ are defined as in (\ref{eqn:20}) and (\ref{eqn:9}). Then
 \begin{enumerate}[(a)]
\item  \begin{align}g(N)&<4T(N)+40\nonumber\end{align} 
 for all real number $N\ge 40$. 
 \item   Further, if $T(N)\le \frac{1}{36} \A(N)$ for all integers $N\ge N_0$, then there exists $N_2\ge N_0$ such that 
  \begin{align}g(N)&\le \psi\left(\frac{N}{2}\right)\nonumber\end{align} for all real number $N\ge N_2$.
 \end{enumerate}
 
\end{lemma}
\begin{proof}
 We know $S_k\le T(N)$ for $k\le m(N)$ and $S_k\le \frac{k^2}{2}$ for $k\ge m(N)$.\\
 Now \begin{align*}g(N)&=1+4(1-e^{-\frac{2}{N}})\{\sum_{k\le m(N)}S_ke^{-\frac{2k}{N}}+\sum_{k>m(N)}S_ke^{-\frac{2k}{N}}\}\\
 &=1+4(1-e^{-\frac{2}{N}})\{\Sigma_3+\Sigma_4\},\quad \text{say.}\\
 \end{align*}
 \begin{align*}
  \Sigma_3&\le\sum_{k=0}^{\infty}T(N)e^{-\frac{2k}{N}}=T(N)\frac{1}{1-e^{-\frac{2}{N}}}.\\
  \Sigma_4&\le \sum_{k>m(N)}\frac{k^2}{2}e^{-\frac{2k}{N}} \le \int_{m(N)}^{+\infty}\frac{x^2}{2} e^{-\frac{2x}{N}}\, dx \le 10
 \end{align*}
and hence \emph{(a)}. 
To prove \emph{(b)} note that $$g(N)\le \frac{1}{9}\A(N)+10$$ for $N\ge 100$. Then, the fact that 
$$ \Psi(\frac{N}{2})>\sum_{\underset{a\le N}{a\in \A}}e^{-\frac{a}{N}}>e^{-2}\sum_{\underset{a\le N}{a\in \A}}1=\frac{1}{e^2}\A(N) $$ proves the result.
\end{proof}

\emph{Lemma \ref{lm_5}} shows condition (\ref{eqn:13}) of \emph{Theorem \ref{Th_2}} are satisfied. Hence                  
 \begin{align}\frac{\psi(N)}{N}&\ge  \exp{\left(-\frac{2.3}{2N}\left(\log_2N+\frac{16}{N}\sum_{k=1}^{\infty}S_k^+\frac{e^{-\frac{2k}{N}}}{1-e^{-\frac{2k}{N}}}\right)-\frac{c}{N}\right)}\label{eqn:44}\end{align}
for some constant $c$ depending on $\A$. 

Taking logarithm on both sides, \begin{align*}{\frac{2.3}{2N}\left(\log_2N+\frac{16}{N}\sum_{k=1}^{\infty}S_k^+\frac{e^{-\frac{2k}{N}}}{1-e^{-\frac{2k}{N}}}\right)+\frac{c}{N}}&>\log \left(1-\left(1-\frac{\psi(N)}{N}\right)\right)>\left(1-\frac{\psi(N)}{N}\right).\end{align*}
Or, \begin{align*}{\frac{2.3}{2N}\left(\log_2N+\frac{16}{N}\sum_{k=1}^{m(N)}S_k^+\left(\frac{e^{-\frac{2k}{N}}}{1-e^{-\frac{2k}{N}}}\right)+10\right)+\frac{c}{N}}&>\left(1-\frac{\psi(N)}{N}\right).\end{align*}
Now, $\frac{e^{-x}}{1-e^{-x}}\le \frac{1}{x}$ and hence we can replace $\frac{e^{-\frac{2k}{N}}}{1-e^{-\frac{2k}{N}}}$ by $\frac{2k}{N}$.
\begin{align}\frac{2.3}{2N}\left(\log_2N+8\sum_{k=1}^{+\infty}\frac{S_k^+}{k}+10\right)+\frac{c}{N}&\ge \frac{1}{e}\left(\frac{N-\A(N)}{N} \right).\nonumber\end{align}
It implies that \begin{align}\s_{k=1}^{+\infty}\frac{S_k^+}{k}>\frac{1}{10e}(N-\A(N))-\frac{1}{8}\log_2 N-c_1 \label{eqn:46}\end{align} 
for large enough $N$ and fixed constant $c_1$ depending on $\A$. This proves \emph{Theorem \ref{Th_1}}.
\section{Monotonicity of $R_1(n)$ on dense set of integers}
\begin{rem}\label{Rm_2}
 \emph{Also, we solved a question raised by S\'{a}rk\"{o}zy (see \cite{10})[Problem 5, Page 337]. His question was the following:}\\
Does there exist an infinite set $\mathcal{A}\subset \mathbb{N}$ such that $\N\setminus \mathcal{A} $ is also an infinite and $R_1(n+1)\ge R_1(n)$ holds on a sequence of integers $n$ whose density is $1$?\\
\emph{Here we show that the answer to this question is positive by giving a simple example.}
\end{rem}
A Sidon set is a set of positive integers such that the sums of any two terms are all different. i.e., $R_2(n)\le 1$ for the corresponding $R_2$ function. By \cite{1}, it is possible to construct sidon sequence of order $(n\log n)^{\frac{1}{3}}$.\\ \\
Now, let $\mathcal{B}$ be an infinite Sidon set of even integers and $\mathcal{A}=\mathbb{N} \setminus \mathcal{B}; $\\
Put $$Y=(\mathcal{B}+\mathcal{B})\cup \mathcal{B}\ \ \text{ and }X=\mathbb{N}\setminus Y; $$
Then, $$R_1(n+1)\ge R_1(n) \quad \text{for all } n\in X.$$\\
To see this, let 
\begin{align*}
 f(z)=\sum_{a\in \mathcal{A}}z^a \quad &\text{and } g(z)=\sum_{b\in \mathcal{B}}z^b.\end{align*}
Then, \begin{align*} \sum_{n=1}^{+\infty}(R_1(n)-R_1(n-1))z^n &=(1-z)f(z)^2\\
&=(1-z)(\frac{z}{(1-z)}-g(z))^2\\
&=\frac{z^2}{(1-z)}+(1-z)(g(z))^2-2zg(z).
\end{align*}
Again, let $$r_1(n)=\sum_{\underset{b_i\in \mathcal{B},b_j\in \mathcal{B}}{b_i+b_j=n}}1,$$
So, $R_1(n+1)\ge R_1(n)$ iff coefficient of $z^{n+1}$ in $(1-z)(f(z))^2$ is non negetive.
\begin{align*}
 \text{Now coefficient of }z^{2k} \text{ is }=&\ \ 1+r_1(2k)-r_1(2k-1)-2\chi_{\mathcal{B}}(2k-1) \\
\text{and coefficient of }z^{2k+1}\text{ is }=&\ \ 1+r_1(2k+1)-r_1(2k)-2\chi_{\mathcal{B}}(2k)
\end{align*}
Then, it is clear from the above choice of $X$ and $\mathcal{A}$ that $R_1(n+1)\ge R_1(n)$ for all $n$ in $X$.\\
For example, we can take $\mathcal{B}=\{2,4,8,16,32,....,2^m,.....\}$. Then $\mathcal{B}$ is infinite and $X$ is of density 1.

\end{document}